\documentstyle[amsfonts,amssymb,amsthm,amsmath, color, 12pt]{article}

\title{New bounds on the cardinality of $n$-Hausdorff and $n$-Urysohn spaces}
\author{Maddalena Bonanzinga\footnote{MIFT Department, University of Messina, Italy, mbonanzinga@unime.it .}, Nathan Carlson\footnote{Department of Mathematics, California Lutheran University, USA,  ncarlson@callutheran.edu .}, Davide Giacopello\footnote{MIFT Department, University of Messina, Italy, dagiacopello@unime.it .}}

\date{}
\begin{document}

 \maketitle

\newtheorem{theorem}{Theorem}[section]
\newtheorem{corollary}[theorem]{Corollary}
\newtheorem{question}[theorem]{Question}
\newtheorem{example}[theorem]{Example}
\newtheorem{lemma}[theorem]{Lemma}
\newtheorem{proposition}[theorem]{Proposition}
\newtheorem{property}[theorem]{Property}
\newtheorem{definition}[theorem]{Definition}
\newtheorem{remark}[theorem]{Remark}
\newtheorem{problem}[theorem]{Problem}

\newcommand{\my}[1]{\textcolor{red}{\sf #1}}
\newcommand{\green}[1]{\textcolor{green}{\sf #1}}
\newcommand{\blu}[1]{\textcolor{blue}{\sf #1}}
\newcommand{\violet}[1]{\textcolor{violet}{\sf #1}}

\begin{abstract}
Two new cardinal functions defined in the class of $n$-Hausdorff and $n$-Urysohn spaces that extend pseudocharacter and closed pseudocharacter respectively are introduced. Through these new functions bounds on the cardinality of $n$-Urysohn spaces that represent variations of known results are given. Also properties of $n$-Urysohn $n$-H-closed spaces are proved.
\end{abstract}

{\bf Keywords:} $n$-Hausdorff spaces, $n$-Urysohn spaces, $n$-H-closed space, pseudocharacter.

{\bf AMS Subject Classification:}   54A25, 54D10, 54D20, 03E99.

\section{Introduction}
 
 Throughout the paper $n$ denote an integer greater
 than or equal to $2$. In \cite{B} Bonanzinga defined the \emph{Hausdorff number} $H(X)$ of a topological space $X$ as the least cardinal number $\kappa$ such that for every subset $A\subseteq X$ with $|A| \geq \kappa$ there exists an open neighbourhood $U_{a}$ for every $a \in A$ such that $\bigcap_{a \in A} U_{a} = \emptyset$. A space $X$ is  said to be \emph{$n$-Hausdorff} if $H(X) \leq n$. Of course,  with $|X|\geq 2$, $X$ is Hausdorff iff $H(X)=2$. The \emph{Urysohn number} $U(X)$ of $X$ is the least cardinal number $\kappa$ such that for every subset $A\subseteq X$ with $|A| \geq \kappa$ there exists an open neighbourhood $U_{a}$ for every $a \in A$ such that $\bigcap_{a \in A} \overline{U_{a}} = \emptyset$. A space $X$ is  said \emph{$n$-Urysohn} if $U(X) \leq n$. Of course,  with $|X|\geq 2$, $X$ is Urysohn iff $U(X)=2$ (see \cite{BCM1,BCM2}).

For a subset $A$ of a topological space $X$ we will denote by $[A]^{<\lambda}$  $([A]^\lambda)$ the family of all subsets of $A$ of cardinality $<\lambda (= \lambda)$.

We consider cardinal invariants of topological spaces (see \cite{E, H}) and except for the Hausdorff number $H(X)$ and the Urysohn number $U(X)$, all the other
cardinal functions are multiplied by $\omega$. In particular, given a topological space $X$, we will denote with $d(X)$ its density, $\chi(X)$ its character, $\pi \chi(X)$ its $\pi$-character, $t(X)$ its tightness, $\psi(X)$ its pseudocharacter, $\psi_c(X)$ its closed pseudocharacter, $nw(X)$ its net-weight and $l(X)$ its Lindel\"of number. Let $(X,\tau)$ be a topological space, a family $\mathcal{B}\subseteq\tau$ is called $\psi$-base (or pseudobase) if $\{x\}=\bigcap\{B\in {\cal B}: x \in B\}$ for every $x\in X$.
	If $X$ is $T_1$ we define the cardinal function ``pseudo weight" as follows  $\psi w(X)=\min\{|{\cal B}|\,\,:\,\,{\cal B} \hbox{ is a } \psi\hbox{-base for } X \}+ \omega$.

Recall that, given a space $X$, the cardinal invariant $d_\theta (X)$, closely related to the density $d(X)$, is defined as follows. A subspace
$D\subseteq X$ is $\theta$-dense in $X$ if $D\cap\overline U\neq\emptyset$ for every non-empty open set $U$ of $X$. The $\theta$-density $d_\theta(X)$ is the least cardinality
of a $\theta$-dense subspace of $X$. Observe that $d_\theta(X)\leq d(X)$ for any space $X$
(see \cite{BCMP} and \cite{C0} for more details about $\theta$-density).

 In \cite[Example 14]{BCMP} it is shown that in the inequality $|X| \leq d(X)^{\chi(X)}$ which is true
for all Hausdorff spaces, one can not replace $d$ with $d_\theta$ (again, for
Hausdorff spaces) and the following is proved.

\begin{theorem}\rm\cite{BCMP}\label{BCMP} If $X$ is $n$-Urysohn space, then $|X| \leq  d_\theta (X)^{\chi(X)}$.
\end{theorem}

	For a space $X$,
	the $n$-Hausdorff pseudocharacter of $X$ (where $n\geq 2$), denoted by $n$-$H\psi(X)$ is the minimum infinite cardinal $\kappa$ such that for each point $x$ there is a collection $\{V(\alpha, x): \alpha< \kappa\}$ of open neighbourhoods of $x$ such that if $x_1,...,x_n$ are different points in $X$ then there exist $\alpha_1,..., \alpha_n<\kappa$ such that $V(\alpha_1, x_1)\cap...\cap V(\alpha_n, x_n)=\emptyset$ \cite{B}. This function is defined for $n$-Hausdorff spaces. Of course, with $|X|\geq 2$, $2$-Hausdorff pseudocharacter of $X$ is exactly the Hausdorff pseudocharacter of $X$, denoted by $H\psi(X)$ \cite{H1991}.

It is well known that the following chain holds for every Hausdorff space $X$
$$ \psi(X)\leq \psi_c(X)\leq H\psi(X)\leq\chi(X).$$

Of course, $H\psi(X)\leq\kappa\Rightarrow n$-$H\psi(X)\leq\kappa$. Also there exist spaces having countable $n$-Hausdorff pseudocharacter which are not first countable \cite{B}.\\

Recall the following proposition
\begin{theorem}\rm \cite[Proposition 30]{B}
	If $X$ is an $n$-Hausdorff space then $|X|\leq d(X)^{n\hbox{-}H\psi(X)}$.
\end{theorem}

	For a space $X$ the $n$-Urysohn pseudocharacter of $X$, denoted by  $n$-$U\psi(X)$ is the minimum infinite cardinal $\kappa$ such that for each point $x$ there is a collection $\{V(\alpha, x): \alpha< \kappa\}$ of open neighbourhoods of $x$ such that if $x_1,...,x_n$ are different points in $X$ then there exist $\alpha_1,..., \alpha_n<\kappa$ such that $\overline{V(\alpha_1, x_1)}\cap...\cap \overline{V(\alpha_n, x_n)}=\emptyset$ \cite{BCP}. This function is defined for $n$-Urysohn spaces. Of course, $2$-Urysohn pseudocharacter of $X$ is exactly the Urysohn pseudocharacter of $X$, denoted by $U\psi(X)$, and introduced in \cite{S}.
The following holds for every Urysohn space:
$$ \psi(X)\leq \psi_c(X)\leq H\psi(X)\leq U\psi(X)\leq\chi(X).$$
Of course, $U\psi(X)\leq\kappa\Rightarrow n$-$U\psi(X)\leq\kappa$.

The almost Lindel\"of degree with respect to closed sets, $aL_c(X)$ is the supremum of $aL(E,X)$ where $E$ is a closed subset of the space $X$ and $aL(E,X)$ is the minimum cardinal number $\kappa$ such that for every family $\cal U$ of open subsets of $X$ that covers $E$ there exist a subfamily $\cal V$ of $\cal U$ such that $|\cal V|\leq \kappa$ and $E\subseteq\bigcup\{\overline{V}: V\in {\cal V}\}$ \cite{WD}. The weak Lindel\"{o}f degree of a space $X$, $wL(X)$, is the minimum cardinal number $\kappa$ such that for every open cover $\cal U$ of $X$ there exist a subfamily $\cal V$ of $\cal U$ such that $|\cal V|\leq \kappa$ and $X=\overline{\bigcup{\cal V}}$.\\

In Section \ref{S2} we introduce a new cardinal function, denoted by $n$-$\psi(\cdot)$ and defined for $n$-Hausdorff spaces, which is an extension of pseudocharacter $\psi(\cdot)$, defined for $T_1$ spaces. Note that $n$-Hausdorffness and satisfying $T_1$ separation axiom are independent properties. We prove that $|X|\leq nw(X)^{n\hbox{-}\psi(X)}$ for each $n$-Hausdorff space $X$ which is true for $T_1$ spaces with $\psi(X)$ in the place of $n\hbox{-}\psi(X)$. Also we obtain that for $T_1$ $n$-Hausdorff spaces the inequality $|X|\leq\psi w(X)^{l(X)n\hbox{-}\psi(X)}$ holds.

In Section \ref{S3} we introduce a new cardinal function, denoted by $n$-$\psi_c(\cdot)$ and defined for $n$-Urysohn spaces which is an extension of the cardinal function $\psi_c(\cdot)$, defined for Hausdorff spaces. Using this new cardinal function we prove that $|X|\leq d(X)^{n\hbox{-}\psi_c(X)wt(X)}$ for every $n$-Urysohn space $X$ where $wt(\cdot)$ is a weaker form of tightness introduced  by Carlson in \cite{C}, and the following modification for $n$-Urysohn spaces of the Willard-Dissanayake's inequality: $|X|\leq d(X)^{\pi\chi(X)n\hbox{-}\psi_c(X)}$. Also we show that if $X$ is an $n$-Urysohn space, then $|X| \leq  \pi\chi(X)^{c(X)\chi(X)}$ and we give a partial answer to the question whether the inequality $|X| \leq  \pi\chi(X)^{c(X)n\hbox{-}\psi_c(X)}$ is true for every $n$-Urysohn space.

In Section \ref{S4} we deal with the class of $n$-H-closed spaces and we give some properties of $n$-Urysohn $n$-H-closed spaces. Recall that an $n$-Hausdorff space is called $n$-$H$-closed if it is closed in every $n$-Hausdorff space in which it is embedded \cite{BBCP0}. In particular we notice that if $X$ is an $n$-Urysohn $n$-H-closed space, then $|X|\leq 2^{n\hbox{-}U\psi(X)}$ hence, $X\leq2^{\chi(X)}$). In addition, since $|X|\leq 2^{\psi_c(X)}$ for every H-closed space $X$ \cite{DP}, we pose the question whether $|X|\leq 2^{n\hbox{-}\psi_c(X)}$ is true for $n$-Urysohn $n$-H-closed spaces. Also we prove that $|X|\leq wL(X)^{\chi(X)}$ holds in the class of $n$-Urysohn locally $n$-H-closed spaces.

\section{The $n$-pseudocharacter of an $n$-Hausdorff space}\label{S2}

In this section we introduce the following cardinal function which is an extension of pseudocharacter in the class of $n$-Hausdorff spaces.
\begin{definition}\rm
	For a space $X$,
	the $n$-\emph{pseudocharacter} of $X$ (where $n\geq 2$), denoted by $n$-$\psi(X)$, is the minimum infinite cardinal $\kappa$ such that for each point $x$ there is a collection $\{V(\alpha,x)\,:\, \alpha<\kappa\}$ of open neighbourhoods of $x$ such that if $x_1,...,x_n\in X$ are $n$ different points, then $\bigcap_{\alpha<\kappa}V(\alpha,x_1)\cap.... \cap \bigcap_{\alpha<\kappa} V(\alpha,x_n)=\emptyset$.
	In other words,
	$n$-$\psi(X)$ is the minimum infinite cardinal $\kappa$ such that for each point $x\in X$ there exists a $G_{\kappa}$-set $A_x$ containing $x$ such that if $x_1,..., x_n\in X$ are $n$ different points, then $A_{x_1}\cap...\cap A_{x_n}=\emptyset$. Note that the $2$-pseudocharacter is defined for $T_1$ spaces while the $n$-pseudocharacter, $n>2$, is defined for $n$-Hausdorff spaces.
\end{definition}

The $2$-pseudocharacter is exactly the pseudocharacter as the following proposition shows.
\begin{proposition}\rm\label{cap-Hpsi(X)=psi(X)}
	Let $X$ be a $T_1$ space, then $2$-$\psi(X)= \psi(X)$.
\end{proposition}
\begin{proof}
	The proof of $2$-$\psi(X)\leq \psi(X)$ is trivial. 
	For every $x\in X$ let $\{V(\alpha,x)\,:\,\alpha<\kappa\}$ be a family of open neighbourhood of $x$ such that if $x,y\in X$ are two different points, then $\bigcap_{\alpha<\kappa}V(\alpha,x)\cap\bigcap_{\alpha<\kappa}V(\alpha,y)=\emptyset$. We want to prove that $\{x\}=\bigcap_{\alpha<\kappa}V(\alpha,x)$. By contradiction let $y\in \bigcap_{\alpha<\kappa}V(\alpha,x)$ be a point different from $x$. Then if we consider $\bigcap_{\alpha<\kappa}V(\alpha,x)\cap\bigcap_{\alpha<\kappa}V(\alpha,y)$ it is certainly nonempty because it contains $y$. So we have a contradiction because $\bigcap_{\alpha<\kappa}V(\alpha,x)\cap\bigcap_{\alpha<\kappa}V(\alpha,y)=\emptyset$.
\end{proof}

It is straightforward to prove the following proposition.
\begin{proposition}\rm
	If $X$ is an $n$-Hausdorff space, $n\geq 2$, then $n$-$\psi(X)\leq n\hbox{-}H\psi(X)$.
\end{proposition}
Recall that any Cantor cube $2^{\kappa}$
	has a dense subset $X$ with $|X|=\kappa$ and $\psi(X) = \omega$. For sake of completeness, we give the contruction of the space $X$.
Let $\{s_\alpha:\alpha<\kappa\}$ be an enumeration
of finite nonempty partial functions from $\kappa$ to 2.
Recursively over $\alpha$, construct a
sequence of countable partial functions
$\{p_\alpha:\alpha\in\kappa\}$ from $\kappa$ to $2$ such that (1) $p_\alpha$ extends $s_\alpha$ and 
(2) $p_\alpha^{-1}(1)\setminus s_\alpha^{-1}(1)$ is an infinite subset of    
	$\kappa\setminus\bigcup_{\beta<\alpha} p_\beta^{-1}(1) $.
Now let $f_\alpha\in 2^\kappa$ be an extension of $p_\alpha$
such that $f_\alpha(\beta)=0$ for all $\beta\not\in dom(p_\alpha)$.
It is clear that $F=\{f_\alpha:\alpha<\kappa\}$ is a dense subset of $2^\kappa$ of cardinality $\kappa$. Also $F$ has countable pseudocharacter. Indeed, it follows from (2) that
$\{f_\alpha\}=F \cap \bigcap\{U_\beta : p_\alpha(\beta)=1\}$,
where $U_\beta=\{x\in 2^\kappa: x(\beta)=1\}$.

Further, recall that in the inequality $|X| \le 2^{c(X)\chi(X)}$ for any Hausdorff space $X$ \cite{HJ, H1984}, character can be replaced with Hausdorff pseudocharacter in
the way similar to the proof of the same theorem (see \cite[Theorem 51]{B}).

Then, since cellularity is hereditary with respect to dense subspaces, it follows that $c(X)=\omega$. By the inequality  $|X| \le 2^{c(X)H\psi(X)}$ for any Hausdorff space, we have that if $\kappa >2^\omega$  any dense subset $X$ of $2^{\kappa}$, with $|X|=\kappa$ and $\psi(X) = \omega$, is an example of a Tychonoff space having countable pseudocharacter and uncountable Hausdorff pseudocharacter. Note that, $\omega=\psi(X)<H\psi(X)\leq\chi(X)=\kappa$; recently, in \cite{BCS} Bella, Carlson and Spadaro constructed a Hausdorff space $X$ such that $\psi(X)<H\psi(X)<\chi(X)$ and asked if there exists a regular space distinguishing the three cardinal functions. Our example is Tychonoff and at least distinguishes $\psi(X)$ and $H\psi(X)$, giving a partial answer to this question. \\

Recall the following result.
\begin{theorem}\rm\cite{B}\label{ex.3-Hausdorff} For a 3-Hausdorff space X, $|X|\leq 2^{2^{c(X) 3-H\psi(X)}}$.
\end{theorem}

Theorem \ref{ex.3-Hausdorff} can be extended to any $n$-Haudorff space in the following way.

\begin{theorem}\rm\label{ex.n-Hausdorff} For a n-Hausdorff space X, $|X|\leq 2^{2^{\cdot^{\cdot^{2^{c(X) n-H\psi(X)}}}}}$, where the power is made ($n-1$)-many times.
\end{theorem}

Then we obtain the following example.

\begin{example}\rm\label{es}
	A Tychonoff (hence $n$-Hausdorff) space $X$ such that $n$-$\psi(X)< n\hbox{-}H\psi(X)$. 
\end{example}

Consider $n=3$. Let $\kappa >2^{2^\omega}$. Any dense subset $X$ of $2^{\kappa}$ with $|X|=\kappa$ and $\psi(X) = \omega$, is an example of a Tychonoff space such that $n$-$\psi(X)=\omega$ for any $n$ and, since $c(X)=\omega$, by Theorem \ref{ex.3-Hausdorff} we have that $3\hbox{-}H\psi(X)>\omega$. A similar proof can be done for every $n$ using Theorem \ref{ex.n-Hausdorff} and taking $\kappa >2^{2^{\cdot^{\cdot^{2^{\omega}}}}}$, where the power is made ($n-1$)-many times. $\hfill$$\triangle$

 \bigskip

Recall the following result

\begin{theorem}\rm \cite[2.3(a)]{J}
	Let $X$ be a $T_1$ space, then $|X|\leq nw(X)^{\psi(X)}$.
\end{theorem}

Now we prove (see Theorem \ref{nw} below) that in the previous result 
${\psi(X)}$ can be replaces by $n$-${\psi(X)}$ when $X$ is $n$-Hausdorff instead of $T_1$. Note that $n$-Hausdorffness and satisfying $T_1$ separation axiom are independent properties (see also \cite{B}). Indeed, consider the set $X=\omega\cup \{p\}\cup\{*\}$  topologized as follows: $\omega$ has the discrete topology, an open basic neighbourhood of $p$ is of the form $U_p=\{p\}\cup(\omega\setminus F)$ where $F$ is a finite subset of $\omega$ and an open neighbourhood of $*$ is of the form $U_*=\{*\}\cup \{p\}\cup(\omega\setminus  F)$, where $F$ is a finite subset of $\omega$. $X$ is a $3$-Hausdorff not $T_1$ space. Moreover, consider the set $Y=4\cup (4\times \omega)$, where $4=\{0,1,2,3\}$, topologized as follows: $4\times \omega$ has the discrete topology, a basic neighbourhood of $i=0,1,2,3$ is of the form $U_i=\{i\}\cup\{j\}\times(\omega\setminus F)$, where $F$ is a finite subset of $\omega$ and $j\not=i$. $Y$ is a $T_1$ not $3$-Hausdorff space.

\begin{theorem}\rm\label{nw}
	Let $X$ be an $n$-Hausdorff space, then $|X|\leq nw(X)^{n\hbox{-}\psi(X)}$.
\end{theorem}
\begin{proof}
	Let $\kappa=n\hbox{-}\psi(X)$ and let $\mathcal{N}$ be a network such that $|{\cal N}|= nw(X)$. \\
	Since $n\hbox{-}\psi(X)= \kappa$, for each $x\in X$ we can fix a collection $\{V(\alpha,x)\,:\, \alpha<\kappa\}$ of open neighbourhoods of $x$ such that if $x_1,...,x_n\in X$ are $n$ different points, then $\bigcap_{i=1}^n\bigcap\{ V(\alpha,x_i)\,:\,\alpha<\kappa\}=\emptyset$.
	Let $x\in X$. Then for every $\alpha<\kappa$ there exists $K(\alpha, x)\in \mathcal{N}$ such that $x\in K(\alpha,x)\subseteq V(\alpha,x)$. Denote by ${\cal K}(x)=\{K(\alpha,x)\,\,:\,\, \alpha<\kappa\}$. Consider the mapping $\phi:X\to [{\mathcal{N}}]^{\leq \kappa}$ defined by $x\mapsto{\cal K}(x)$. We want to prove that the cardinality of each fiber is strictly less than $n$. This will imply that $|X|\leq nw(X)^\kappa$. Let $x_1,...,x_n$ be $n$ different points in $X$. Then $\bigcap_{i=1}^n\bigcap \{V(\alpha,x_i) : \alpha <\kappa\}=\emptyset$, therefore $\bigcap_{i=1}^n\bigcap \{K(\alpha,x_i)\,:\alpha<\kappa \}=\emptyset$. So we have that there exists a $j\in\{1,..,n\}$ such that ${\cal K}(x_j)\not={\cal K}(x_i)$ for some $i\not=j$. 
\end{proof}

\begin{proposition}\rm \label{psiw} \cite[Proposition 2.3]{J}
	For a $T_1$ space $X$, $nw(X)\leq \psi w(X)^{l(X)}$.
\end{proposition}

Combining Propositions \ref{nw} and \ref{psiw} we obtain

\begin{corollary}\rm
	For a $T_1$ $n$-Hausdorff space $X$, $|X|\leq\psi w(X)^{l(X)n\hbox{-}\psi(X)}$.
\end{corollary}

The authors were not able to find the source of the result following result and we give the proof for sake of completness.

\begin{proposition}\rm 
\label{J}
	For a Hausdorff space $X$, $\psi w(X)\leq d(X)^{t(X)}$.
\end{proposition}
\begin{proof} 
	Let $D$ be a dense subset of $X$ such that $|D|\leq d(X)$. Let $\kappa=t(X)$ and ${\cal D}=\{\overline{A}: A\in [D]^{\leq\kappa}\}$. We have that $X=\overline{D}=\bigcup {\cal D}$ and $|{\cal D}|\leq d(X)^\kappa$. Put ${\cal B}=\{X\setminus \overline{A}: \overline{A}\in {\cal D}\}$. We will prove that it is a $\psi$-base. Let $x, y\in X$, then there exists an open subset $U$ of $X$ such that $x\not\in \overline{U} $ and $y\in U$. Since $D$ is dense and $t(X)\leq \kappa$, $y\in \overline{U\cap D}$ and there exists a subset $A\in [U\cap D]^{\leq \kappa}$ such that $y\in \overline{A}$. Therefore $y\not\in X\setminus \overline{A}$ and $x\in X\setminus \overline{A}$.
\end{proof}

\begin{question}\rm\label{Q} Is it true that $\psi w(X)\leq d(X)^{t(X)}$ 
	(or, at least, $\psi w(X)\leq d(X)^{t(X)l(X)n\hbox{-}\psi(X)}$), for an n-Hausdorff space $X$?
\end{question}

\begin{remark}\rm
	Note that, by Propositions \ref{nw} and \ref{psiw}, a positive answer to the previous question will allow us to prove that $|X|\leq d(X)^{t(X)l(X)n\hbox{-}\psi(X)}$ for every $T_1$ $n$-Hausdorff space $X$ and then to obtain a partial answer to the following question posed in \cite{B}.
\end{remark}

\begin{question}\rm\label{QB} \cite{B}
	Is it true that if $X$ is a $T_1$ $n$-Hausdorff space, then $|X|\leq2^{l(X)t(X)\psi(X)}$?
\end{question}
	
\section{$n$-closed-pseudocharacter of an $n$-Urysohn space} \label{S3}

In this section we introduce the following cardinal function which is a generalization of the closed pseudocharacter in the class of $n$-Urysohn spaces.
\begin{definition}\rm For a space $X$,
	the $n$-\emph{closed pseudocharacter	of} $X$ (where $n\geq 2$) denoted by $n$-$\psi_c(X)$, is the minimum infinite cardinal $\kappa$ such that for each point $x$ there is a collection $\{V(\alpha,x)\,:\, \alpha<\kappa\}$ of open neighbourhoods of $x$ such that if $x_1,...,x_n\in X$ are $n$ different points, then $\bigcap \{\overline{V(\alpha,x_1)}\,:\alpha<\kappa \}\cap.... \cap\bigcap \{\overline{V(\alpha,x_n)}\,:\alpha<\kappa \}=\emptyset$. Note that the $2$-closed pseudocharacter is defined for Hausdorff spaces while the $n$-closed pseudocharacter is defined for $n$-Urysohn spaces.
\end{definition}

It is straightforward to prove, as in Proposition \ref{cap-Hpsi(X)=psi(X)}, that the $2$-closed pseudocharacter is exactly the closed pseudocharacter.

\begin{proposition}\rm
	$2$-$\psi_c(X)= \psi_c(X)$, for every Hausdorff space $X$. 
\end{proposition}
Also, it follows directly from the definitions that the following inequality holds.
\begin{proposition}\rm
	$n\hbox{-}\psi_c(X)\leq n\hbox{-}U\psi(X)$, for every $n$-Uryshon space $X$.
\end{proposition}

\begin{example}\rm
	A Tychonoff (hence $n$-Urysohn) space such that $n\hbox{-}\psi_c(X)< n\hbox{-}U\psi(X)$.
\end{example}

Consider the space $X$ of Example \ref{es}. Recall that pseudocharacter and closed pseudocharacter coincide for regular spaces and that $n$-$H\psi(Y) \leq n$-$U\psi(Y)$, for every space $Y. \hfill \triangle$

\bigskip

Recall the following result.

\begin{theorem}\rm\cite{BC}
	If $X$ is Hausdorff, then $|X|\leq d(X)^{\psi_c(X)t(X)}$.
\end{theorem}
In \cite{C} Carlson generalized the previous result introducing the following weaker form of tightness. 

\begin{definition}\rm\cite{C}
	Let $X$ be a space. The weak tightness $wt(X)$ of $X$ is defined as the least infinite cardinal $\kappa$ for which there is a cover ${\cal C}$ of $X$ such that $|{\cal C}|\leq 2^{\kappa}$ and for all $C\in \mathcal C$, $t(C)\leq \kappa$ and $X=\bigcup_{B\in [C]^{\leq 2^\kappa}}\overline{B}$. 
\end{definition}

It is clear that $wt(X)\leq t(X)$.

\begin{theorem}\rm\cite{C}
	If $X$ is Hausdorff, then $|X|\leq d(X)^{\psi_c(X)wt(X)}$.
\end{theorem}

We prove a variation (Theorem \ref{gen} below) of the previous theorem in the class of $n$-Urysohn spaces. First we prove the following general result.

\begin{theorem}\rm\label{new}
	Let $\kappa$ and $\lambda$ be two infinite cardinals. Let $X$ be an n-Urysohn space such that
	\begin{itemize}
		\item[1.] for every $p\in X$ there exists a family ${\mathcal V}_p$ of open neighbourhoods of $p$ such that $|{\mathcal V}_p|\leq \kappa$ and for every $x_1,...,x_n\in X$, $(\bigcap_{V\in {\mathcal V}_{x_1}} \overline{V})\cap.... \cap(\bigcap_{V\in {\mathcal V}_{x_n}} \overline{V})=\emptyset$.
		\item[2.] there exists a subset $S$ of $X$ such that $|S|\leq \lambda$ and $X=\bigcup_{A\in [S]^{\leq\kappa}}\overline{A}$.
	\end{itemize}
	Then $|X|\leq \lambda^\kappa$.
\end{theorem}
\begin{proof}
	For every $p\in X$ fix $A_p\in[S]^{\leq\kappa}$ such that $p\in\overline{A}_p$. For every $V\in {\mathcal V}_p$, we have that $p\in\overline{A_p\cap V}$ and $A_p\cap V\in [S]^{\leq\kappa}$. Define the function $\phi: X\to [[S]^{\leq\kappa}]^{\leq\kappa}$ such that $\phi(p)=\{A_p\cap V: V\in{\mathcal V}_p\}$ for every $p\in X$. We will show that $\phi$ is a $n-1$ to $1$ map; this implies that $|X|\leq |[[S]^{\leq\kappa}]^{\leq\kappa}|\leq (|S|^\kappa)^\kappa\leq \lambda^\kappa$. Assume, by contradiction, there exist $x_1,...,x_n\in X$ such that $\phi(x_1)=...=\phi(x_n)$. Then $\{A_{x_1}\cap V: V\in{\mathcal V}_{x_1}\}=...=\{A_{x_n}\cap V: V\in{\mathcal V}_{x_n}\}$. Therefore $(\bigcap_{V\in {\mathcal V}_{x_1}} \overline{V})\cap.... \cap(\bigcap_{V\in {\mathcal V}_{x_n}} \overline{V})\not=\emptyset$; a contradiction.
\end{proof}

In \cite{JVM2018}, Juh{\'a}sz and van Mill introduced the notion of a $\mathcal{C}$-saturated subset of a space $X$. 
\begin{definition}\rm
Given a cover $\mathcal{C}$ of $X$, a subset $A\subseteq X$ is $\mathcal{C}$-\emph{saturated} if $A\cap C$ is dense in $A$ for every $C\in\mathcal{C}$. 
\end{definition}
It is clear that the union of $\mathcal{C}$-saturated subsets is $\mathcal{C}$-saturated. The following is given in~\cite{JVM2018} in the case $\kappa=\omega$, and extended to the general case in~\cite{C}.

\begin{lemma}\rm\cite{JVM2018,C}\label{saturated}
Let $X$ be a space, $wt(X)=\kappa$, and let $\mathcal{C}$ be a cover witnessing that $wt(X)=\kappa$. Then for all $x\in X$ there exists $S_x\in[X]^{\leq 2^\kappa}$ such that $x\in S_x$ and $S_x$ is $\mathcal{C}$-saturated.
\end{lemma}

\begin{theorem}\rm\label{gen}
	If $X$ is an $n$-Urysohn space, then $|X|\leq d(X)^{n\hbox{-}\psi_c(X)wt(X)}$.
\end{theorem}
\begin{proof}
	Put $\kappa=n$-$\psi_c(X)wt(X)$. Let $\mathcal C$ be a cover witnessing that $wt(X)\leq\kappa$ and $D$ be a dense subset of $X$ such that $|D|=d(X)$.  By Lemma~\ref{saturated}, for all $x\in X$ there exists a $\mathcal{C}$-saturated set $S_x$ such that $x\in S_x$ and $|S_x|\leq 2^\kappa$. Let $S=\bigcup_{d\in D}S_d$. Then $S$ is $\mathcal{C}$-saturated, as $S$ is the union of $\mathcal{C}$-saturated sets, and $|S|\leq |D|\cdot 2^\kappa\leq|D|^\kappa$. Observe that as $D\subseteq S$, we have that $S$ is dense in $X$. Since $n$-$\psi_c(X)\leq\kappa$, for every $p\in X$ fix a collection $\{V(\alpha,p)\,:\, \alpha\leq\kappa\}$ of open neighbourhoods of $p$ such that if $x_1,...,x_n\in X$ are $n$ different points, then $\bigcap \{\overline{V(\alpha,x_1)}\,:\alpha<\kappa \}\cap.... \cap\bigcap \{\overline{V(\alpha,x_n)}\,:\alpha<\kappa \}=\emptyset$.
	Fix $x\in X$ and let $C\in\mathcal{C}$ such that $x\in C$. We will show that for each $\alpha<\kappa$ we have $x\in \overline{V(\alpha,x)\cap S\cap C}^{C}$.
	As $S$ is $\mathcal{C}$-saturated (hence $S\cap C$ is dense in $S$) and $S$ is dense in $X$, we have $S\cap C$ is dense in $X$ and then
	$x\in V(\alpha,x) \subseteq \overline{V(\alpha,x)}\subseteq \overline{V(\alpha,x)\cap C\cap S}.$
Therefore $x\in C\cap \overline{V(\alpha,x)\cap S\cap C}=\overline{V(\alpha,x)\cap S\cap C}^{C}$. As $t(C)\leq\kappa$, there exists $A_\alpha\subseteq V(\alpha,x)\cap S\cap C$ such that $x\in \overline{A_\alpha}^{C}\subseteq\overline{A_\alpha}$ and $|A_\alpha|\leq\kappa$. Now let $A_x=\bigcup_{\alpha<\kappa}A_\alpha$ and note $x\in\overline{A_x}$, $|A_x|\leq\kappa\cdot\kappa=\kappa$, and $A_x\subseteq S$.

Unfixing $x$, we see that $X=\bigcup_{x\in X}\overline{A_x}=\bigcup_{A\in[S]^{\leq\kappa}}\overline{A}$. Letting $\lambda=d(X)^\kappa$, we see that the conditions of Theorem \ref{new} hold and thus $|X|\leq \lambda^\kappa= (d(X)^\kappa)^\kappa=d(X)^\kappa$. 
\end{proof}

\begin{corollary}\rm\label{d}
	If $X$ is an $n$-Urysohn space, then $|X|\leq d(X)^{n\hbox{-}\psi_c(X)t(X)}$.
\end{corollary}
Since $n\hbox{-}\psi_c(X)t(X)\leq \chi(X)$, we can obtain the following corollary.

\begin{corollary}\rm\label{chi}
	If $X$ is an $n$-Urysohn space, then $|X|\leq d(X)^{\chi(X)}$.
\end{corollary}

Considering Theorem \ref{BCMP}, it is natural to pose the following question.
Recall that the $\theta$-closure of a subset $A$ of a space $X$, denoted by $cl_\theta(A)$, is the subset $\{x\in X: \overline{U}\cap A\not= \emptyset \hbox{ for every open neighbourhood } U \hbox{ of }x \}$ and the $t_\theta (X)$ is the minimum cardinal number $\kappa$ such that for every $x\in X$ and every subset $A$ of $X$ such that $x\in cl_\theta(A)$ there exists a subset $B$ such that $x\in cl_\theta(B)$ and $|B|\leq \kappa$. 
\begin{question}\rm
	Is it true that $|X|\leq d_{\theta}(X)^{n\hbox{-}\psi_c(X)t_\theta(X)}$ for every $n$-Urysohn space $X$?
\end{question}

\bigskip

Now we give an $n$-Urysohn generalization of the following Willard-Dissanayake result.

\begin{theorem}\rm \cite{WD} If $X$ is a Hausdorff space then $|X|\leq d(X)^{\pi\chi(X)\psi_c(X)}$.
\end{theorem}

\begin{theorem}\rm If $X$ is an $n$-Urysohn space then $|X|\leq d(X)^{\pi\chi(X)n\hbox{-}\psi_c(X)}$.
\end{theorem}
\begin{proof}
Let $\kappa=\pi\chi(X)n\hbox{-}\psi_c(X)$ and let $D$ be a dense set such that $|D|=d(X)$. For all $x\in X$, let $\{V(\alpha,x): \alpha<\kappa\}$ be a closed $n$-pseudobase at $x$ and let ${\cal U}_x$ be a local $\pi$-base at $x$ such that $|{\cal U}_x|\leq\kappa$. Fix $x\in X$. Let ${\cal U}(\alpha,x)=\{U\in{\cal U}_x:U\subseteq V(\alpha,x)\}$. For all $U\in{\cal U}(\alpha,x)$, let $d(U,\alpha,x)\in U\cap D$. Let $D(\alpha,x)=\{d(U,\alpha,x):U\in{\cal U}_x\}$. Note $|D(\alpha,x)|\leq |{\cal U}_x|\leq\kappa$. We show that $x\in \overline{D(\alpha,x)}$ for every $\alpha<\kappa$. Let $x\in W$, where $W$ is open. Then $x\in W\cap V(\alpha,x)$ for every $\alpha<\kappa$. As ${\cal U}_x$ is a local $\pi$-base at $x$, there exists $U_\alpha\in{\cal U}_x$ such that $U_\alpha\subseteq W\cap V(\alpha,x)$ and thus $U_\alpha\in{\cal U}(\alpha,x)$ for every $\alpha<\kappa$. It follows that $d(U_\alpha,\alpha,x)\in U_\alpha\subseteq W\cap V(\alpha,x)$ and thus $W\cap D(\alpha,x)\not=\emptyset$. This shows that $x\in \overline{D(\alpha,x)}\subseteq\overline{V(\alpha,x)}$ for every $\alpha<\kappa$. Consider the function $\phi: X\to [[D]^{\leq \kappa}]^{\leq\kappa}$ such that for each $x\in X$, $\phi(x)=\{D(\alpha,x):\alpha<\kappa\}$. This is a $(n-1)\hbox{-}1$ map, in fact, suppose, by contradiction, that there exist $x_1,...,x_n$ different points in $X$ such that $\{D(\alpha,x_1)\}_{\alpha<\kappa}=...=\{D(\alpha,x_n)\}_{\alpha<\kappa}$. Since $\bigcap_{\alpha<\kappa}\overline{D(\alpha,x_i)}$ is not empty for each $i=1,...,n$, $\bigcap_{\alpha<\kappa}\overline{D(\alpha,x_1)}=...=\bigcap_{\alpha<\kappa}\overline{D(\alpha,x_n)}=E$ and $\bigcap_{\alpha<\kappa}\overline{D(\alpha,x_i)}\subseteq \bigcap_{\alpha<\kappa}\overline{V(\alpha,x_i)}$ for every $i=1,...,n$, we have that $\emptyset\not=E\subseteq \bigcap_{i=1}^{n}\bigcap_{\alpha<\kappa}\overline{V(\alpha,x_i)}$. That is a contradiction. 
\end{proof}

\bigskip

Recall the following result.

\begin{theorem}\rm\cite{Sun}\label{Sun}
	If X is a Hausdorff space, then
	$|X| \leq  \pi\chi(X)^{c(X)\psi_c(X)}$.
\end{theorem}

Carlson proved the following.

\begin{theorem}\rm\cite{C0}\label{C0} For any space $X$, $d_\theta(X)\leq\pi\chi(X)^{c(X)}$.
\end{theorem}

Then by Theorems \ref{BCMP} and \ref{C0}, we have the following result.

\begin{theorem}\rm\label{nU} If $X$ is $n$-Urysohn, then $|X| \leq  \pi\chi(X)^{c(X)\chi(X)}$.
\end{theorem}
Note that recently in \cite{BCGM} it is proved that $|X|\leq2^{c(X)\pi\chi(X)} $, for every $n$-Urysohn homogeneous space $X$. It is natural to pose the following question.

\begin{question}\rm\label{Que} Is $|X| \leq  \pi\chi(X)^{c(X)n\hbox{-}\psi_c(X)}$ true for every $n$-Urysohn space?
\end{question}

Theorem \ref{pQ} below gives a partial answer to the previous question in the class of $n$-Urysohn quasiregular spaces. Recall that a space is said to be quasiregular if for each open set $U$ there exists a open subset $V$ of $U$ contained with its closure in $U$. The following examples show that quasiregularity and $n$-Urysohness are independent properties.

\begin{example}\rm
	A quasiregular space which is not $3$-Urysohn.
\end{example}

	Consider a space similar to the one of Example \ref{ex.3-Hausdorff}. $X=3\cup (\omega\times 3)$, where $3=\{0,1,2\}$.
	The points from $\omega\times 3$ has the discrete topology, instead, a basic open neighbourhood of $i=0,1,2$ is $U(i,F)=\{i\}\cup \{(\omega\times j)\setminus (F\times j)\,\,:\,\,j\in\{0,1,2\}\,\,,\,\,j\not=i\}$, where $F$ is a finite subset of $\omega$. $X$ is not $3$-Urysohn since $0,1$ and $2$ belong to the closure of each open neighbourhood of them. 
$\hfill$ $\triangle$

Carlson and Ridderbos in \cite{CR}, under $(\frak c^+=2^{\frak c})$, constructed an example of a Urysohn ccc space which has the $\pi$-character equal to the continuum and the density equal to the successor of the continuum. Then such space cannot be quasiregular since the inequality  $d(X)\leq\pi\chi(X)^{c(X)}$, proved by Sapirovski for regular spaces \cite{S}, holds also for quasiregular spaces, that is
\begin{theorem}\rm\cite{C0}\label{quasi} For a quasiregular space $X$, $d(X)\leq\pi\chi(X)^{c(X)}$.
\end{theorem}
The previous theorem follows from Theorem \ref{C0} and the fact that $d$ and $d_\theta$ coincide in the class of quasiregular spaces.\\
\\
Then, by Theorems \ref{gen} and \ref{quasi}, we obtain the following partial answer to Question \ref{Que}. 

\begin{theorem}\rm\label{pQ}
	If $X$ is an $n$-Urysohn quasiregular space, then $|X| \leq \pi\chi(X)^{c(X)wt(X)n\hbox{-}\psi_c(X)}$.
\end{theorem}


\bigskip

Recall that every Hausdorff space having a compact $\pi$-base is quasiregular. In \cite{Tka83}, Tkachenko introduced the o-tightness of a space, and in \cite{BCG} Bella, Carlson and Gotchev used it to give an improvement of Theorem \ref{Sun} in the class of Hausdorff spaces having a compact $\pi$-base. We recall that the o-tightness of a space $X$ does not exceed $\kappa$, or $ot(X) \leq\kappa$, if
for every family $\mathcal U$ of open subsets of $X$ and for every point $x\in X$ with $x\in \overline{\bigcup{\mathcal U}}$
there exists a subfamily $\mathcal{V}\subseteq\mathcal U$ such that $|{\mathcal V}|\leq\kappa$ and $x\in  \overline{\bigcup{\mathcal V}}$.
We have that $ot(X)\leq t(X)$ and $ot(X)\leq c(X)$.
Also $wL(X)ot(X)\leq c(X)$. Moreover $ot(X)\leq wt(X)$, for any space $X$ \cite{BCG}.

\begin{theorem}\rm\cite{BCG} If X is a Hausdorff space with a compact $\pi$-base, then
	$|X| \leq\pi\chi(X)^{wL(X)ot(X)\psi_c(X)}$.
\end{theorem}

It is natural to pose the following question.

\begin{question}\rm Is $|X| \leq  \pi\chi(X)^{wL(X)ot(X)n\hbox{-}\psi_c(X)}$ true for every $n$-Urysohn (quasiregular) space?
\end{question}

\section{On $n$-Urysohn $n$-H-closed spaces}\label{S4}
In this section we present some properties of $n$-Urysohn $n$-H-closed spaces motivated by similar properties that hold in the class of H-closed spaces. 
Recall that a space is said to be H-closed if it is Hausdorff and it is closed in every Hausdorff space in which it is embedded. An $n$-Hausdorff space is called $n$-$H$-closed if it is closed in every $n$-Hausdorff space in which it is embedded \cite{BBCP0}.\\

\begin{theorem}\rm \cite{BBCP0}
	For an $n$-Hausdorff space $X$ the following are equivalent:
	\begin{itemize}
		\item[1.] $X$ is $n$-H-closed;
		\item[2.] for each open ultrafilter $\mathcal{U}$ on $X$, $|a({\cal U})|=n-1$, where $a(\cal U)$ denotes the adherence of $\cal U$;
		\item[3.] for every open filter ${\cal F}$ on $X$, $|a({\cal F})|\geq n-1$;
		\item[4.] for every $A\in [X]^{<n-1}$ and for each family of open subsets ${\cal U}$ of $X$ such that $X\setminus A\subseteq \bigcup{\cal U}$, there exists a finite subfamily ${\cal V}$ of ${\cal U}$ such that $X=\bigcup_{V\in {\cal V}}\overline{V}$.
	\end{itemize}
\end{theorem}
For a topological space $(X, \tau)$ and $U \in \tau$, let $r(U)$ denote $int_X(\overline{U})$ and $X(s)$ denote the space whose set is $X$ and whose topology is generated by the base  $\{r(U):U \in \tau\}$.
A function $f:X\to Y$ is called $\theta$-continuous at $x_0$ if for every open neighbourhood $V$ of $f(x_0)$ in $Y$, there exists an open neighbourhood $U$ of $x_0$ in $X$ such that $f(\overline{U}^X)\subseteq \overline{V}^Y$.\\

\smallskip
Recall the following result.
\begin{proposition}\rm\cite{PW}
	Let $f:X\to Y$ be a $\theta$-continuous surjection. If $X$ is a H-closed space, then $Y$ is H-closed.
\end{proposition}

Now we prove the following.

\begin{proposition}\rm\label{theta}
	Let $f:X\to Y$ be a $\theta$-continuous bijection. If $X$ is an $n$-H-closed space then $Y$ is an $n$-H-closed space.
\end{proposition}
\begin{proof}
	Fix $A\in [Y]^{< n-1}$ and a family $\mathcal U$ of open subsets of $Y$ with $Y\setminus A\subseteq \bigcup \mathcal U$. For every $x\in X\setminus f^{-1}(A)$ there exists an open subset $V(x)\in \mathcal U$ such that $f(x)\in V(x)$. Since $f$ is $\theta$-continuous, there is an open neighbourhood $U(x)$ of $x$ in $X$ such that $f(\overline{U(x)}^X)\subseteq \overline{V(x)}^Y$. Considering that the space $X$ is $n$-H-closed, there exists a finite subset $F$ of $X$ such that $X=\bigcup_{x\in F} \overline{U(x)}^X$. Since $Y=f(X)=f(\bigcup_{x\in F} \overline{U(x)}^X)=\bigcup_{x\in F}f(\overline{U(x)}^X)\subseteq \bigcup_{x\in F}\overline{V(x)}^Y$ holds, $Y$ is $n$-H-closed.
\end{proof}
In \cite{PW} it is proved that the identity $id: X(s)\to X$ is a $\theta$-homeomorphism. This fact allows to prove the following proposition.
\begin{proposition}\rm\cite{PW}\label{T2}
	A space $X$ is H-closed iff $X(s)$ is H-closed.
\end{proposition} 
So, using Proposition \ref{theta} and the same fact, we can prove the following.
\begin{proposition}\rm\label{T4}
	A space $X$ is $n$-H-closed iff $X(s)$ is $n$-H-closed.
\end{proposition}

\begin{proposition}\rm\cite{PW}
	Let $X$ be a H-closed space and $U$ be an open subset of $X$. Then $\overline{U}$ is H-closed.
\end{proposition}
We can prove that the previous proposition can be extended to $n$-H-closed spaces.

\begin{proposition}\rm
	Let $X$ be an $n$-H-closed space and $U$ be an open subset of $X$. Then $\overline{U}$ is $n$-H-closed.
\end{proposition}
\begin{proof}
	Note that the $n$-Hausdorff property is hereditary. Let $A=\overline{U}^X$ and $\mathcal F$ an open filter on $A$. Then $\{F\cap U: F\in\mathcal F\}$ is an open filter base on $X$. Let ${\cal G}=\{W\subseteq X: W \hbox{ is open in } X \hbox{ and } F\cap U\subseteq W \hbox{ for some }F\in\mathcal F\}$. Then $\mathcal G$ is an open filter on $X$, therefore $|a_X({\cal G})|\geq n-1$. Since $a_X({\cal G})\subseteq \bigcap_{F\in {\cal F}}\overline{(F\cap U)}^X= \bigcap_{F\in {\cal F}}\overline{(F\cap U)}^A\subseteq \bigcap_{F\in {\cal F}}\overline{F}^A=a_A({\cal F}) $, we have that $|a_A({\cal F})|\geq |a_X({\cal G})|\geq n-1$.
\end{proof}

We prove the following surprising result.

\begin{proposition}\rm
	Every $n$-Hausdorff $T_1$ space with at least one isolated point is not $n$-H-closed. 
\end{proposition}
\begin{proof}
	Suppose by contradiction that for each $A\in [X]^{n-1}$ and every family of open subsets $\cal U$ of $X$ with $X\setminus A\subseteq \bigcup {\cal U}$ there exists a finite subfamily $\cal V$ of $\cal U$ such that $X=\bigcup\{\overline{V}: V\in {\cal V}\}$. Suppose there exists one isolated point $x_1$. Let $A=\{x_1\}$. Since $X$ is $T_1$ there exists a family $\cal U$ of open subsets such that $X\setminus A=\bigcup {\cal U}$. Since $x_1$ is isolated then it does not belong to $\overline{U}$ for each $U\in {\cal U}$, therefore one cannot select a subfamily $\cal V$ of $\cal U$ such that $X=\bigcup\{\overline{V}: V\in {\cal V}\}$. 
\end{proof}
Recall that $\psi_c(X) = \psi_c(X(s))$, for any Hausdorff space $X$. Now we prove the following result.
\begin{proposition}\rm
	If $X$ is an $n$-Urysohn space then $n\hbox{-}\psi_c(X)=n\hbox{-}\psi_c(X(s))$.
\end{proposition}
\begin{proof}
	Since $\overline{U}=\overline{r(U)}=\overline{r(U)}^{X(s)}$ (see \cite{PW}), if $X$ is $n$-Urysohn, then $X(s)$ is $n$-Urysohn and $n\hbox{-}\psi_c(X)=n\hbox{-}\psi_c(X(s))$.
\end{proof}
 
The following result is assumed without proof in \cite{DP}; we give the proof for sake of completeness.

\begin{lemma}\rm\label{p1}
	Let $X$ be a H-closed space. Then $\chi(X(s))\leq\psi_c(X)$.
\end{lemma}
\begin{proof}
	Let $\kappa = \psi_c(X)$ and $x\in X$.  There is a family $\mathcal{U}$ of open neighbourhood of $x$ of such that $x \in \bigcap_{U\in \mathcal{U}} U \subseteq \bigcap_{U\in \mathcal{U}}\overline{U}=\{x\}$ and $\kappa = |\mathcal{U}|$. Without loss of generality we can assume that $\mathcal U$ is closed under finite intersections. We want to show that $\{int(\overline{U}):U \in \mathcal{U}\}$ is a neighborhood base of $x$ in $X(s)$. Let $T$ be an open neighborhood of $x$ in $X(s)$.  As $X(s)$ is semiregular, we can assume that $T = int(\overline{U})$ is regular open.  
		So, $ \{x\}= \bigcap_{U\in\mathcal{U}}\overline{U}\subseteq T$ 
		and then $X\backslash T\subseteq  X\backslash \bigcap_{U\in\mathcal{U}}\overline{U}= \bigcup_{U\in\mathcal{U}}X\backslash\overline{U}$.
		Thus, $\{X\backslash \overline{U} :U \in \mathcal{U}\}$ is a family of regular open sets of $X$ that cover $X\backslash T$. Since $X\backslash T$ is a H-set (i.e. a regular closed subset in a H-closed space), 
		there is a finite subset $\mathcal{G} \subseteq \mathcal{U}$ such that $X\backslash T \subseteq \bigcup_{U\in \mathcal{G}} \overline{X\backslash \overline{U}}=\bigcup_{U\in \mathcal{G}}X\setminus int(\overline{U})=X\setminus \bigcap_{U\in \mathcal{G}}int(\overline{U})$.
		Then, $\bigcap_{U\in \mathcal{G}}int(\overline{U})
		\subseteq T$ implying $x \in int(\overline{\bigcap_{U\in \mathcal{G}} U})\subseteq int(\bigcap_{U\in \mathcal{G}}\overline{U})=\bigcap_{U \in \mathcal{G}}int(\overline{U}) \subseteq T$. By the arbitrarity of $x$, we conclude that $\chi(X(s)) \leq \kappa$.
\end{proof}
Since the inequalities $\psi_c(X)\leq \psi_c(X(s))\leq\chi(X(s))$ are true for every space $X$, by the previous lemma, we obtain the following proposition.

\begin{proposition}\rm\cite{DP}\label{T1}
	Let $X$ be a H-closed space. Then $\chi(X(s))=\psi_c(X)$.
\end{proposition}

Then it is natural to pose the following question.
\begin{question}\rm\label{q1}
	Let $X$ be an $n$-Urysohn $n$-H-closed space. Is $\chi(X(s))\leq n$-$\psi_c(X)$ true?
\end{question}

Recall the following amazing Dow and Porter's result that strongly improved the Gryzlov's theorem: \textit{if $X$ is a H-closed space with $\psi_c(X)\leq \omega$, then $|X|\leq 2^\omega$}\cite{Gr}.
\begin{theorem}\rm\cite{DP}\label{c1}
	Let $X$ be a H-closed space. Then $|X|\leq 2^{\psi_c(X)}$.
\end{theorem}
The previous theorem was proved using Propositions \ref{T2}, \ref{T1} and the next result:

\begin{theorem}\rm\cite{DP}\label{T3} Let $X$ be a H-closed space. Then $|X|\leq 2^{\chi(X)}$.
\end{theorem}

Theorem \ref{T3} can be extended to the class of $n$-H-closed spaces.

\begin{theorem}\rm\cite{BBCP}\label{n-H-closed}
	Let $X$ be an $n$-H-closed space. Then $|X|\leq 2^{\chi(X)}$.
\end{theorem}

Then it is natural to pose also the following question.

\begin{question}\rm\label{question}
	Is it true that $|X|\leq 2^{n\hbox{-}\psi_c(X)}$ (or at least $|X|\leq 2^{n\hbox{-}\psi_c(X)t(X)}$) for every $n$-Urysohn $n$-H-closed space?
\end{question}

Using the following result, we give a partial answer (Theorem \ref{T5} below) to Question \ref{question} in the class of $n$-Urysohn $n$-H-closed spaces.
	\begin{theorem}\rm\cite{BCP}
		If $X$ is an $n$-Urysohn space. Then $|X|\leq 2^{aL(X)n\hbox{-}U\psi(X)}$.
	\end{theorem}
	
	\begin{theorem}\rm\label{T5} Let $X$ be an $n$-Urysohn $n$-H-closed space. Then $|X|\leq 2^{n\hbox{-}U\psi(X)}$ (hence $|X|\leq 2^{\chi(X)}$).
\end{theorem}

\begin{remark}\rm
	Note that, by Theorems \ref{T4} and \ref{n-H-closed}, a positive answer to Question \ref{q1} allows us to give a positive answer to Question \ref{question} too.
\end{remark}

Finally, we give a bound for the cardinality of $n$-Urysohn, locally $n$-H-closed spaces. This result is an analogue of Theorem 4.2 in \cite{BCG1} which involves locally H-closed spaces. Recall that a space is called locally $n$-H-closed if for every point $x$ there exists an open neighbourhood which closure is $n$-H-closed. 
\begin{theorem}\rm
	Let $X$ be an $n$-Urysohn locally $n$-H-closed space. Then $|X|\leq wL(X)^{\chi(X)}$.
\end{theorem}
\begin{proof}
	For every $x\in X$ there exists an open neighbourhood $U_x$ of $x$ such that $\overline{U_x}$ is $n$-H-closed. By Theorem \ref{n-H-closed}, $|\overline{U_x}|\leq 2^{\chi(\overline{U_x})}\leq 2^{\chi(X)}$. Clearly, ${\cal U}=\{U_x:x\in X\}$ is an open cover of $X$, then there exists ${\cal V}\in [{\cal U}]^{\leq wL(X)}$ such that $X=\overline{\bigcup {\cal V}}$. The set $\bigcup {\cal V}$ is dense in $X$ and $|\bigcup{\cal V}|\leq wL(X)2^{\chi(X)}\leq wL(X)^{\chi(X)}$. By Corollary \ref{chi} we can conclude the proof.
\end{proof} 

\bigskip

{\bf Acknowledgement:} The authors express gratitude to I. Juhasz and L. Zdomskyy for useful discussions. The research was supported by \lq\lq National Group for Algebric and Geometric Structures, and their Applications" (GNSAGA-INdAM). 

\end{document}